\documentclass[12pt]{amsart}

\usepackage{amsmath, amsthm, amssymb}

\usepackage{url} % <-- Para p\'aginas web o similar: \url{...}
\usepackage{graphicx}
\usepackage{xcolor}
\usepackage{moreverb}

\usepackage{fancyvrb}

\def\r{\mathbb R}
\def\m{\mathbb M}
\def\s{\mathbb S}
\def\h{\mathbb H}

\def\v{\mathcal V}

\def\w{\mathcal W}
\def\x{\mathfrak X}
\DeclareMathOperator{\sech}{sech}

\usepackage{listings}
 \lstnewenvironment{code}[1][]
  {\lstset{#1}}% Add/update settings locally
  {}

\newtheorem{theorem}{Theorem}[section]
 \newtheorem{proposition}[theorem]{Proposition}
 \newtheorem{corollary}[theorem]{Corollary}
 \newtheorem{lemma}[theorem]{Lemma}
\theoremstyle{definition}
\newtheorem{definition}[theorem]{Definition}
\newtheorem{example}[theorem]{Example}
\newtheorem{remark}[theorem]{Remark}

\begin{document}

\title[Curves in Riemannian manifolds]{Curves in Riemannian Manifolds Making Prescribed Angles With Torse-Forming Vector Fields}

\author{Muhittin Evren Aydin$^1$}
\address{$^1$Department of Mathematics, Faculty of Science, Firat University, Elazığ,  23200 Türkiye
\newline
ORCID: 0000-0001-9337-8165}
\email{meaydin@firat.edu.tr}
\author{ Esra Dilmen$^2$}
 \address{$^2$Firat University, Graduate School of Natural and Applied Sciences, Mathematics,
Elazığ, Türkiye
\newline
ORCID: 0000-0001-6750-9475}

 \email{esradlmn23@gmail.com}
\author{ Büşra Karakaya$^3$}
 \address{$^3$Firat University, Graduate School of Natural and Applied Sciences, Mathematics,
Elazığ, Türkiye
 \newline
ORCID: 0009-0000-6285-3701}
 \email{busrakarakaya404@gmail.com}

\begin{abstract}
In this paper, we introduce the notion of a prescribed angle curve in a Riemannian manifold associated with a pair $(\mathcal{V},\theta)$, where $\mathcal{V}$ is a unit vector field along the curve and $\theta$ denotes the angle between $\mathcal{V}$ and the principal normal vector of the curve. When $\mathcal{V}$ is a torse-forming vector field, we establish an existence result for prescribed angle curves. In the $3$-dimensional case, we determine the curvatures of these curves in terms of the prescribed angle and the potential function of $\mathcal{V}$. Moreover, using this notion, we provide a new characterization of curves lying on geodesic spheres in real space forms.
\end{abstract}
\maketitle

%%%%%%%%%%%%%%%%%%%%%%%%%%%%%%%%
\section{Introduction} \label{intro}
%%%%%%%%%%%%%%%%%%%%%%%%%%%%%%%%

In classical differential geometry, one of the common problems is the study of curves in Riemannian manifolds which make a prescribed angle between one of their Frenet vectors and a given vector field $\v$. Since many classes of curves were investigated from this perspective, the literature is extensive, as we will discuss below.

In the particular case of the Euclidean space, it is natural to assume that $\v$ is a parallel vector field, that is, a fixed direction $\vec{v}\in \r^m$. In this case $\vec{v}$ is called an {\it axis} of the given curve. One of the most well-known examples of such curves is the general helix (respectively, slant helix), that is, a curve whose tangent (respectively, principal normal) makes a constant angle with $\vec{v}$ \cite{it}. The classical Lancret theorem states that a curve is a general helix if and only if the ratio of its torsion to its curvature is constant \cite{lan}.

When extending the notion of helix to more general ambient spaces, the choice of an appropriate axis becomes an issue, as clearly emphasized in \cite{dgd}. To our knowledge, the first attempt in this manner was made by Hayden in 1930s \cite{hay}, assuming that the axis is a vector field $\v$ which is parallel along the curve. In more recent work \cite{dgd}, the authors performed this idea to study parallel slant helices in arbitrary Riemannian manifolds.

In another work, Barros \cite{bar} assumed that the axis is a vector field $\v$ which is Killing along the curve and obtained Lancret-type theorems in the sphere $\s^3(1)$ and the hyperbolic space $\h^3(-1)$. Following this idea, Lucas and Ortega-Yag\"ues \cite{loy2} defined and characterized slant helices in $\s^3(1)$. Nistor \cite{nis} classified curves in the homogenous space $\s^2\times \r$ whose tangent makes constant angle with Killling vector fields. More recently, L\'opez \cite{lo} provided a complete classification of general helices in $\r^3$ and $\h^3(-1)$ corresponding to all types of Killing vector fields.

In almost contact metric geometry, slant curves (respectively, $N$-slant curves) are defined as curves whose tangent (respectively, principal normal) makes a constant angle with the Reeb vector field \cite{agy,ccm,cil}. The term was also used for curves in a warped product $I \times_\rho \r^2$ whose tangent makes a constant angle with the vector field tangent to the interval $I$ \cite{ccra}. See also \cite{cs}.

In the early 2000s, Chen \cite{crc00} introduced the so-called rectifying curves in $\r^3$ whose principal normal is always perpendicular to the position vector. This notion has since been extended to other ambient spaces and higher dimensions \cite{dgd0,ggb,loy01,loy1}.

Denote by $\m^m(c)$ one the standard models of real space forms, i.e. $\r^m$, $\s^m(c)$, or $\h^m(-c)$. Recently, Lucas and Ortega-Yag\"ues \cite{loy4,loy5} defined a new class of curves, called concircular helices in $\m^m(c)$. Before proceeding, we recall that a vector field $\v$ on a Riemannian manifold $M^m$ is said to be {\it torse-forming} if it satisfies
\begin{equation*}
\nabla_X\v=fX+\omega(X) \v, \quad X\in \mathfrak{X}(M^m),
\end{equation*}
where $\nabla$ is the Levi-Civita connection on $M^m$, $f:M^m\to \r$ is a smooth function and $\omega$ is a $1$-form of $M^m$ \cite{ya0}. The function $f$ and $1$-form $\omega$ are called the potential function and generating form of $\v$, respectively. See also \cite{bo1,id,mih0}. When both $f$ and $\omega$ vanish identically, the vector field $\v$ becomes parallel trivially. This case is excluded from the present study.

There are several important subclasses of such vector fields. A torse-forming vector field $\v$ is said to be concircular if $\omega \equiv 0$  \cite{ya1}. Let $\w$ denote the vector field dual to the $1$-form $\omega$. Then, the vector field $\v$ is called torqued (respectively, anti-torqued) if $\w \perp \v$ (respectively, $\w = -f\v$) \cite{crs3,cs,dan,na}. These vector fields also have applications in physics, see \cite{c0,ds,mam1}.

Returning to the concept of a concircular helix, such a curve in the space $\m^m(c)$ is defined by the condition that the function $\langle N, \v \rangle$ is constant along the curve, where $N$ is the unit principal normal vector field of the curve and $\v$ is a globally defined concircular vector field on the ambient space \cite{loy4,loy5}. In the cited works, these curves were also classified. Following the same idea, the first author, together with Mihai and Özgür, introduced and characterized the notions of torqued and anti-torqued curves \cite{amc2}.

Denote by $\theta$ the angle function between the vectors $N$ and $\v$ along the curve. The condition that $\langle N, \v \rangle$ is a constant $k\in \r$, $k\neq 0$, imply that $\theta$ is constant only if the vector $\v$ has constant length along the curve. Notice here that a concircular vector field never has constant length on a whole Riemannian manifold (see \cite[Proposition 3.2]{amc}). Otherwise, the function $\theta=\cos^{-1}(k|\v|^{-1})$ is non-constant, although it is prescribed.

Curves whose principal normal vector makes a prescribed, but not necessarily constant, angle with a given vector field are of interest not only in differential geometry but also in geometric analysis. Certain examples are soliton solutions to the curve shortening flow \cite{kl,lo1}, curves minimizing the moment of inertia \cite{eul,lo3} and curves minimizing the gravitational potential energy \cite{die,lo2}.

By motivated the above definitions, we introduce the following.

\begin{definition}\label{int-def}
Let $\gamma:I\subset \r \to M^m$ be a Frenet curve in a connected orientable Riemannian $m$-manifold $M^m$, parametrized by arc length. Let $\v$ be a nowhere zero vector field along $\gamma$ and $\theta:I\to[0,\pi]$ be a smooth function. We say that $\gamma$ is a prescribed angle curve associated with the pair $(\v,\theta)$ if $\v$ has unit length along $\gamma$ and $\theta$ is the angle between $\v$ and the principal normal vector field  of $\gamma$.
\end{definition}

For brevity, such a curve will be called a PA curve with $(\v,\theta)$. When the dependence on $\v$ and $\theta$ is not emphasized, we simply call it a PA curve. 

In this paper, we consider PA curves associated with non-parallel torse-forming vector fields. Notice that the condition in Definition \ref{int-def} that the vector field $\v$ has constant length along the curve  is natural; see \cite{bar,dgd,lo,loy6,loy7}. Moreover, we prescribe the angle between $\v$ and the principal normal vector of the curve, rather than the tangent vector. This is because, when $\v$ is concircular, it is always orthogonal to the curve (see Corollary \ref{thm44}); hence prescribing the angle with the tangent vector would lead to a restrictive situation.

Assuming the existence of a vector field defined on the entire ambient space imposes strong restrictions on the geometry and topology of the space. For instance, no the space $\m^m(c)$ admits a globally defined torqued vector field whose potential function and generating form are both nowhere zero (see \cite[Theorem 4.1]{amc}, \cite[Theorems 1, 2]{dta}). 

Moreover, even assuming the existence of a vector field only along a curve may still impose significant constraints. For example, in the studies \cite{bar,lo,loy2}, the vector field considered as the axis is assumed to be Killing along the curve. However, in the space $\m^m(c)$, any vector field that is Killing along a curve extends to a Killing vector field on the entire ambient space (see \cite[Lemma 1]{bar}). Since Riemannian manifolds, in general, do not admit rich symmetry structures, the class of ambient spaces suitable for such studies is limited.

In contrast, this situation does not hold for torse-forming vector fields. To illustrate this, consider the unit circle in $\r^2$ given by $\gamma(s)=(\cos s,\sin s)$ and define a vector field along $\gamma$ by $\v(s)=\int^{s} (\cos u \,\gamma'(u))\,du.$ Then, $\v'(s)=\cos s\,\gamma'(s),$
which means that $\v$ is concircular along $\gamma$ with $\cos s$ potential function. However, it is known that the potential function of a globally defined concircular vector field on $\r^2$ must be constant (see \cite{loy4}). Hence, $\v$ cannot be extended to a concircular vector field on $\r^2$. 

This shows that, as opposed to Killing vector fields, assuming the existence of a torse-forming vector field along a curve provides a wider freedom in the choice of ambient spaces. Another illustration of this situation is the following: as emphasized above, although no the space $\m^m(c)$ admits a globally defined torqued vector field, such vector fields exist on suitable open subsets (see \cite[Example 3.1]{amc} and \cite[Example 1]{dta}).

We would like to express the advantages of the notion of a PA curve introduced in Definition \ref{int-def} as follows. Clearly, along a Frenet curve, the angle between a nowhere zero vector field $\v$ and the principal normal vector is well defined. However, imposing conditions on this angle, not necessarily requiring it to be constant, facilitates a better understanding of the geometric properties of the curve. For instance, the curvatures of such a curve can be determined only in terms of the prescribed angle and the potential function of the torse-forming vector field $\v$ (see Theorems \ref{thm42} and \ref{thm31}, Corollaries \ref{thm44} and \ref{lanc}).

Definition \ref{int-def} also provides new perspectives on the geometric behavior of curves. More explicitly, using the notion of PA curve, we give a new characterization of curves lying on geodesic spheres in the space $\m^m(c)$. We obtain that a curve lies on a geodesic sphere if and only if there exists a unit concircular vector field $\v$ along the curve with potential function $f=(a\tan \theta +b)^{-1}$, $a,b\in\r$, where $\theta$ is  the angle function between $\v$ and the principal normal vector (Theorem \ref{thm44spherical}). 

Furthermore, PA curves associated with concircular vector fields extend the Lancret theorem to arbitrary Riemannian manifolds (Corollary \ref{lanc}).

The organization of paper is the following. In Section \ref{pre}, we recall the Frenet formulas in Riemannian manifolds and give examples of globally defined torse-forming vector fields on warped products and real space forms. In Section \ref{sechig}, after establishing an existence result for PA curves, we determine the curvatures of PA curves in Riemannian manifolds of dimension greater than $2$ in terms of the potential function and the prescribed angle. In the final section, Section \ref{secdim2}, we study such curves in Riemannian surfaces. In both Sections \ref{sechig} and \ref{secdim2}, several illustrative examples are provided.

%%%%%%%%%%%%%%%%%%%%%%%%%%%%%%%%
\section{Preliminaries} \label{pre}
%%%%%%%%%%%%%%%%%%%%%%%%%%%%%%%%

Let $M^m$ be a connected orientable Riemannian $m$-manifold and let $\gamma:\,I\to M^m$ be a smooth curve, where $I\subset \r$ is an interval.  The curve $\gamma$ is said to be a Frenet curve if the vector fields $\{\gamma',\nabla_{\gamma'}\gamma',...,\nabla_{\gamma'}^{m-1}\gamma'\}$ are linearly independent along $I$, where $\nabla$ is the Levi-Civita connection on $M^m$.

If in addition, the curve $\gamma$ is parametrized by arc length $s\in I$, then there exist uniquely determined vector fields $\{X_1,\ldots,X_m\}$ along $\gamma$ and functions $\kappa_i:\,I\to \r$, with $\kappa_i >0$ for $1\leq i \leq m-1$, called  the curvatures of $\gamma$, such that:
\begin{enumerate}
\item[(i)] $\{X_1,\ldots,X_m\}$ is a positively oriented orthonormal frame along $\gamma$, called the Frenet frame of $\gamma$, with $X_1=\gamma'$,

\item[(ii)] for each point of $\gamma$, the sets $\{X_1=\gamma',\ldots,X_m\}$ and $\{\gamma',\nabla_{\gamma'}\gamma',...,\nabla_{\gamma'}^{m-1}\gamma'\}$ 
span the same subspace of $T_{\gamma(s)}M^m$,

\item[(iii)]  the Frenet formulas hold:
\begin{equation*}
\left. 
\begin{array}{l}
\nabla_{\gamma'}\gamma'=\kappa_1 X_2 \\ 
\nabla_{\gamma'}X_i =-\kappa_{j-1} X_{i-1}+\kappa_{i} X_{i+1} \\
\nabla_{\gamma'}X_{m}=-\kappa_{m-1} X_{m-1},
\end{array}%
\right. 
\end{equation*}
for $2\leq i \leq m-1 $ (see [Proposition 3.2]\cite{lmm}).
\end{enumerate}

Conversely, let $\gamma(s_0)=p_0\in M^m$ and $(e_1,...,e_m)$ a positively oriented orthonormal basis of $T_{p_0}M^m$. Giving smooth functions $k_i:\,(s_0-\delta,s_0+\delta)\to \r$, with $k_i >0$ for $1\leq i \leq m-1$, there exists $0< \epsilon \leq \delta$ and a unique Frenet curve $\gamma:\,(s_0-\epsilon,s_0+\epsilon)\to M^m$ such that
\begin{enumerate}
\item[(i)] $\gamma(s_0)=p_0$,
\item[(ii)] $X_i(s_0)=e_i$ for $1\leq i \leq m$,
\item[(iii)] $\kappa_i=k_i$ for $1\leq i \leq m-1$ (see [Theorem 3.6]\cite{lmm}).
\end{enumerate}

From now on, we denote $\gamma'=T$, $X_2=N$ and $X_{i+2}=B_i$, $1\leq i \leq m-2$. 

Since our framework is based on torse-forming vector fields, it is necessary to recall the notion of warped products, as such spaces provide natural examples of these vector fields.

Let $(M,g_1)$ and $(F,g_2)$ be two Riemannian manifolds.  A warped product $M \times_{\rho} F$ is the product manifold $M \times F$ endowed with the metric $\langle , \rangle=g_1+\rho^2g_2,$ 
where $\rho:\,M\to (0,\infty)$ is a smooth function, called the warping function \cite{cbook}.

Let $J\subset \r$ be an interval and consider a warped product $J\times_\rho F$. Denote by $t$ the arc length parameter on $J$ and $\partial_t:=\frac{\partial}{\partial_t}$. Then:
\begin{enumerate}
\item[(i)] the vector field $\v=\rho\partial_t$ is concircular on $J\times_\rho F$ with the potential function $f=\rho'$ (see \cite[Example 1.1]{cbook}).
\item[(ii)] The vector field $\v=h\rho\partial_t$ is torqued on $J\times_\rho F$ with the potential function $f=h$ and the generating form $\omega=dh$, where $h$ is a smooth function depending only on the fiber $F$, i.e. $\partial_t h=0$ \cite[Theorem 2.3] {c01}.
\item[(iii)] The vector field $\v=\partial_t$ is anti-torqued on $J\times_\rho F$ with $f=(\log \rho)'$  \cite[Theorem 3.1]{amc1}.
\end{enumerate}

We now provide examples of torse-forming vector fields on the space $\m^m(c)$.

\begin{enumerate}
\item[(i)] Let the ambient space be $\r^m$. Up to a constant vector field, any concircular vector field is of the form $\v=r\Phi$, where $\Phi$ is the position vector field on $\r^m$ and $r\in \r$ is a nonzero constant (see \cite[Remark 2.1] {c01} and \cite[Proposition 4.2]{c02}). In addition, the vector field $\v=\frac{\Phi}{|\Phi|}$ is anti-torqued on the connected manifold $\r^m\setminus\{0\}$ with the potential function $f=\frac{1}{|\Phi|}$ \cite[Example 5]{dan}. 
\item[(ii)] Let the ambient space be $\s^m(c)\subset\r^{m+1}$ or $\h^m(-c)\subset\r^{m+1}_1$. Then any concircular vector field is given by the tangential component of a constant vector field $\vec{v}\in \r^{m+1}$ or $\vec{v}\in \r^{m+1}_1$, respectively, where $ \r^{m+1}_1$ denotes the Lorentz-Minkowski space \cite[Theorem 1]{loy5}.
\item[(iii)] Consider the upper half-space model of $\h^m(-1)$, $$\h^m(-1)=\{(x_1,...,x_m)\in \r^m:x_m>0\}$$ endowed with the metric $\langle,\rangle_h=x_m^{-2}\langle,\rangle$, where $\langle,\rangle$ is the Euclidean metric. Then the vector field $\v=-x_m\partial_{x_m}$ is anti-torqued with the potential function $f\equiv1$ \cite[Example 3.3]{amc}.
\end{enumerate}

We finish this section recalling the following result.

\begin{proposition} \cite{cbook} \label{warp-con}
Let $M=M_1 \times_{\rho} M_2$ be a warped product, and let $\nabla^0$ and $\nabla^i$ be the Levi-Civita connections on $M$ and $M_i$, respectively. Let also $X_i,Y_i \in \mathcal{L}(M_i)$. Then,
\begin{enumerate}
\item[(i)]  $\nabla^0_{X_1}Y_1 \in \mathcal{L}(M_1) $ is the lift of $\nabla^1_{X_1}Y_1$ on $M_1$.
\item[(ii)]  $\nabla^0_{X_1}X_2 =\nabla^0_{X_2}X_1 =(X_1 \log \rho)X_2$.
\item[(iii)]  $\textnormal{tan}(\nabla^0_{X_2}Y_2)$ is the lift of $\nabla^2_{X_2}Y_2$ on $M_2$.
\item[(iv)]  $\textnormal{nor}(\nabla^0_{X_2}Y_2)=-\frac{\langle X_2,Y_2\rangle}{\rho}\nabla^0\rho$.
\end{enumerate}
\end{proposition}

%%%%%%%%%%%%%%%%%%%%%%%%%%%%%%%%
\section{PA curves in Riemannian manifolds} \label{sechig}
%%%%%%%%%%%%%%%%%%%%%%%%%%%%%%%%

In this section, we consider PA curves in a Riemannian $m$-manifold $M^m$ associated with non-parallel torse-forming vector fields. 

Let $\gamma:\,I\to M^m$ be a Frenet curve parametrized by arc length. Denote by $\{T,N,B_1,\ldots,B_{m-2}\}$ the Frenet frame of $\gamma$ and by $\kappa_i$, $1\leq i \leq m-1$, its curvatures. Assume that $\v$ is a unit torse-forming vector field along $\gamma$. Then,
\begin{equation}
\nabla_T\v=fT+\omega(T)\v, \label{sec3tor}
\end{equation}
where $f:\,I\to \r$ and $\omega(T):\,I\to \r$ are smooth functions along $\gamma$.

Since we assume that $\v$ has unit length along $\gamma$, it follows that $0=\langle \nabla_T\v,\v\rangle$. Using expression \eqref{sec3tor}, we obtain
\begin{equation}
f\langle \v, T\rangle+\omega(T)=0. \label{cdimm}
\end{equation}
With equation \eqref{cdimm}, we immediately have the following result.

\begin{lemma}\label{consleng}
Let $\gamma $ be a Frenet curve in a Riemannian manifold and $\v$ a unit torse-forming vector field along $\gamma$. Its potential function vanishes along $\gamma$ if and only if $\v$ is parallel along $\gamma$.
\end{lemma}

Assume now that $\gamma$ is a PA curve in $M^m$ with $(\v,\theta)$, where $\v$ is a torse-forming vector field. By its definition, there are smooth functions $\lambda_i:\,I\to \r$, $0\leq i \leq m-2$, defined along $\gamma$ such that
\begin{equation}
\left. 
\begin{array}{l}
\v=\lambda_0 T+\cos\theta N+\lambda_1 B_1+\ldots \lambda_{m-2}B_{m-2}, \\ 
1=\cos^2\theta+\sum_{i=0}^{m-2}\lambda_i^2.
\end{array}%
\right. \label{trs41}
\end{equation}
With this decomposition of $\v$, equation \eqref{cdimm} becomes
\begin{equation}
f\lambda_0+\omega(T)=0. \label{cdimm2}
\end{equation}

Differentiating $\v$ along $\gamma$ in the first expresion of \eqref{trs41} and using the Frenet formulas yields
\begin{equation}
\left. 
\begin{array}{l}
\nabla_T\v= (\lambda'_0-\cos\theta \kappa_1)T\\
+((\cos \theta)'+\lambda_0\kappa_1-\lambda_1\kappa_2)N \\
+(\lambda_1'+\cos\theta\kappa_2-\lambda_2\kappa_3)B_1\\
+(\lambda_i'+\lambda_{i-1}\kappa_{i+1}-\lambda_{i+1}\kappa_{i+2})B_i, \quad (2\leq i \leq m-3),\\
+(\lambda'_{m-2}+\lambda_{m-3}\kappa_{m-1})B_{m-2} .
\end{array}%
\right. \label{trs42}
\end{equation}
On the other hand, since $\v$ is a torse-forming vector field along $\gamma$, we have
$$
\nabla_T\v=(f+\lambda_0\omega(T))T+\omega(T)\cos\theta N+\omega(T)\sum_{i=1}^{m-2}\lambda_iB_{i}.
$$
Comparing this expression with \eqref{trs42} and then using \eqref{cdimm2}, we obtain
\begin{equation}
\left. 
\begin{array}{r}
f(1-\lambda_0^2)-\lambda'_0+ \kappa_1\cos\theta=0\\
-f\lambda_0\cos\theta+\theta'\sin \theta-\lambda_0\kappa_1+\lambda_1\kappa_2=0 \\
-f\lambda_0\lambda_1-\lambda_1'+\lambda_2\kappa_3-\kappa_2\cos\theta=0\\
-f\lambda_0\lambda_i-\lambda_i'-\lambda_{i-1}\kappa_{i+1}+\lambda_{i+1}\kappa_{i+2}=0 \\
-f\lambda_0\lambda_{m-2}-\lambda'_{m-2}-\lambda_{m-3}\kappa_{m-1}=0,
\end{array}%
\right. \label{trs43}
\end{equation}
where $2\leq i \leq m-3$.

Consequently, we have the following existence result of PA curves in Riemannian manifolds.

\begin{theorem}\label{prop41}
Let $\gamma $ be a Frenet curve in a connected orientable Riemannian $m$-manifold. Then, the curve $\gamma$ is a PA curve with $(\v,\theta)$, where $\v$ is a torse-forming vector field with potential function $f$ and generating form $\omega$, if and only if there are smooth functions $\theta,\lambda_0,\ldots\lambda_{m-2}$ along $\gamma$ such that $\v$ is given by \eqref{trs41}, the curvatures of $\gamma$ satisfy the system \eqref{trs43} and $\omega(T)=-f\langle \v,T\rangle$, where $T$ is the unit tangent vector.
\end{theorem}
\begin{proof}
It remains to prove the converse of the statement of the theorem. Let $\kappa_i\,:I\subset \r \to \r$, $1\leq i \leq m-1$ be smooth functions given by \eqref{trs43} and let $\gamma(s_0)=p_0\in M^m$. Then, there exists a unique curve $\gamma$ (see [Theorem 3.6]\cite{lmm})
with curvatures $\kappa_i$ and initial Frenet frame $\{T(s_0),N(s_0),B_1(s_0),\ldots,B_{m-2}(s_0)\}$. Next, define the vector field $\v$ along the curve $\gamma$ as in \eqref{trs41}. Then it has unit length and the angle between $\v$ and $N$ is given by $\theta$. Differentiating $\v$ along $\gamma$ and using the system \eqref{trs43}, we have
$$
\nabla_T\v = fT - f\lambda_0(\lambda_0T+\cos \theta\, N+\lambda_1B_1+\cdots+\lambda_{m-2}B_{m-2}).
$$
Using $\omega(T)=-f\langle \v,T\rangle$, it follows that the covariant derivative $\nabla_T\v $ is of the form \eqref{sec3tor}, which means that $\v$ is a torse-forming vector field along $\gamma$.
\end{proof}

In the particular $3$-dimensional case, writing $B_1=B$, $\kappa_1=\kappa$ and $\kappa_2=\tau$, the decomposition of $\v$ in \eqref{trs41} and the system \eqref{trs43} reduces to, respectively,
\begin{equation}
\v=\lambda_0T+\cos \theta N+\lambda_1 B, \quad |\v |=1, \label{trs44}
\end{equation}
and
\begin{equation}
\left. 
\begin{array}{r}
f(1-\lambda_0^2)-\lambda'_0+ \kappa\cos\theta=0 \\ [4pt]
-f\lambda_0\cos\theta+\theta'\sin \theta-\lambda_0\kappa+\lambda_1\tau=0 \\ [4pt]
f\lambda_0\lambda_{1}+\lambda'_{1}+\cos\theta\tau=0.
\end{array}%
\right. \label{trs45}
\end{equation}

In the following results, we derive expressions for $\kappa$ and $\tau$ in terms of $f$, $\lambda_0$ and $\lambda_1$, depending on whether $\cos \theta = 0$. In particular, the case $\theta = \frac{\pi}{2}$ significantly simplifies the system \eqref{trs45}.

\begin{theorem}\label{thm42}
Let $\gamma $ be a Frenet curve in a connected orientable Riemannian $3$-manifold and $\v$ a torse-forming vector field along $\gamma$ with potential function $f$. Then $\gamma$ is a PA curve with $(\v,\frac{\pi}{2})$ if and only if 
$$
\v=(\tanh p) T\pm (\sech p) B, \quad p:=\int^s f(u)du+r, \, r\in \r.
$$
In addition, the ratio of the curvatures of $\gamma$ is $\frac{\tau}{\kappa}=\sinh (p)$ and in particular this ratio is nonconstant.
\end{theorem}
\begin{proof}
If $\gamma$ is a PA curve with $(\v,\frac{\pi}{2})$, then $\langle\v,N\rangle=0$. By \eqref{trs44}, there is a smooth function $\vartheta$ defined along $\gamma$ such that $\v=\cos \vartheta T+\sin\vartheta B$.
Moreover, the system \eqref{trs45} is now
\begin{equation}
\left. 
\begin{array}{r}
\sin\vartheta(\vartheta'+f\sin\vartheta)=0 \\ [4pt]
\tau\sin\vartheta=\kappa \cos\vartheta \\ [4pt]
\cos\vartheta\left(\vartheta'+f\sin\vartheta\right)=0.
\end{array}%
\right. \label{simp41}
\end{equation}
If $\cos\vartheta=0$, then, by the first equation in \eqref{simp41}, $f$ would vanish along $\gamma$. Hence, from Lemma \ref{consleng} it follows that $\v$ becomes parallel along $\gamma$, which is not our case. Similarly, from the second equation in \eqref{simp41}, the case $\sin\vartheta=0$ is not possible because  $\gamma$ is assumed to be a Frenet curve. Consequently, it follows from \eqref{simp41} that
$$
\vartheta'+f\sin\vartheta=0.
$$
Integrating this equation, we derive $\cos\vartheta=\tanh p$ and $ \sin\vartheta=\sech p.$ Finally, the ratio of $\kappa$ and $\tau$ is concluded by the second equation in \eqref{simp41}, which is non-constant because $f$ is non-vanishing along $\gamma$. The proof of the converse follows from direct calculations and Theorem \ref{prop41}.

\end{proof}

\begin{remark}
Recall from \cite[Theorem 1]{dgd} that if $\v$ is parallel along a curve $\gamma\subset M^3$ and $\langle \v,N\rangle=0$, then the angle $\vartheta$ between the vectors $T$ and $\v$ is constant; hence, the ratio $\tau/\kappa = \cot \vartheta$ is also constant. Such curves are known as parallel generalized helices \cite{dgd}. On the other hand, if $\gamma$ is a PA curve in $M^3$ with $(\v,\frac{\pi}{2})$, where $\v$ is a non-parellel unit torse-forming vector field (so that $\langle \v,N\rangle=0)$, the angle $\vartheta$ between $T$ and $\v$ is never constant. In contrast to the parallel case, although the relation $\tau/\kappa=\cot \vartheta$ still holds, the ratio $\tau/\kappa$ is not constant.
\end{remark}

We provide two examples of PA curves with $(\v,\frac{\pi}{2})$ in the ambient spaces where $\v$ is a globally defined anti-torqued vector field.

\begin{example}
Let $\Phi$ be the position vector field on $\r^3$ and consider the torse-forming vector field $\v=\frac{\Phi}{|\Phi|}$ globally defined on the connected manifold $M^3=\r^3\setminus \{0\}$. Its potential function is given by $f=\frac{1}{|\Phi|}$. Then, each rectifying curve $\gamma\subset M^3$ parametrized by arc length is a PA curve with $(\v,\frac{\pi}{2})$. Indeed,  since $\v$ is a unit anti-torqued vector field on $M^3$, so is along $\gamma$. Moreover, it is known from \cite[Theorem 1]{crc00} that $\gamma(s)=(s+b)T(s)+aB(s)$, $a,b\in\r,$ $a\neq0.$ Since $f(\gamma(s))=f(s)=\frac{1}{\sqrt{(s+b)^2+a^2}}$, we have, up to a suitable constant,
$$
p=\int^sf(\gamma(u))du=\tanh ^{-1}\left((s+b)f(s)\right).
$$
Hence, following the proof of Theorem \ref{thm42}, we verify that
$$
\cos \vartheta(s)=\tanh p=(s+b)f(s), \quad \sin \vartheta(s)=\sech p=af(s),
$$
and  $\frac{\tau}{\kappa}=\cot \vartheta=\frac{s+b}{a},$ which agrees with \cite[Theorem 2]{crc00}. Recall that the parametric equations of all rectifying curves in $\r^3$ are well-known (see \cite[Theorem 3]{crc00}).
\end{example}

\begin{example}
Let the ambient space be the warped product $M^3=J\times_{\rho(t)}\r^2$, where $J\subset \r$ is an interval and $\rho(t)=e^t$. It is known that $\v=\partial_t$ is an anti-torqued vector field with potential function $f(t)=(\log \rho(t))'=1$. Let $\phi(s)$ be a strictly increasing $C^1$ function on an interval $I\subset \r$ and consider the curve $\gamma:I\to M^3$ given by
$$
\gamma(s)=\left ( \log (\cosh s), \int^s \sech^2 u\cos \phi(u) du, \int^s \sech^2 u\sin \phi(u) du\right ).
$$
We will show that $\gamma$ is a PA curve with $(\v,\frac{\pi}{2})$ in $M^3$. It is direct to see that $\gamma$ is parametrized by arc length. Hence,
$$
T(s)=\left ( \tanh s, \sech^2 s\cos \phi(s), \sech^2 s\sin \phi(s)\right ).
$$
Choose the orthonormal frame $\{\partial_t,e_1,e_2\}$ on $M^3$, where $e_1=e^{-t}\partial_x$ and $e_2=e^{-t}\partial_y$. From Proposition \ref{warp-con}, the nonzero covariant derivatives of this frame are given by
$$
\nabla_{e_1}\partial_t=e_1, \quad \nabla_{e_2}\partial_t=e_2, \quad \nabla_{e_1}e_1=\nabla_{e_2}e_2=-\partial_t.
$$
With respect to this frame, we can write
$$
T(s)=\tanh s\, \partial_t + \sech s (\cos \phi(s)\, e_1+\sin \phi(s)\, e_2).
$$
Then
$$
\nabla_TT=\phi'(s)\sech s \, (-\sin \phi(s) \,e_1+\cos \phi(s) \,e_2),
$$
which implies
$$
\kappa(s)=\phi'(s)\sech s, \quad N(s)=-\sin \phi(s) \,e_1+\cos \phi(s) \,e_2.
$$
Thus, $N$ is orthogonal the vector field $\v=\partial_t$, namely, $\gamma$ is a PA curve with $(\partial_t,\frac{\pi}{2})$ in $M^3$. By a computation we obtain
$$
B(s)=\sech s \, \partial_t-\tanh s(\cos \phi(s)\, e_1+\sin \phi(s)\, e_2),
$$
which yields $\v(\gamma(s))=\tanh s\,T(s)+\sech s \, B(s)$. Moreover,
$$
\nabla_TB=-\phi'(s)\tanh s (-\sin \phi(s) \,e_1+\cos \phi(s) \,e_2),
$$
and hence $\tau(s)=\phi'(s)\tanh s $. Finally, $\tau/\kappa=\sinh s$.
\end{example}

\begin{remark}
Curves in $M^3=J\times_{\rho(t)} \r^2$ whose tangent vector makes a constant angle with the vertical vector field $\partial_t$ were completely classified in \cite[Theorem 4.1]{ccra}. These curves are known as slant curves in $M^3$. According to \cite[Proposition 3.3]{ccra}, the normal vector of a slant curve in $M^3$ cannot be orthogonal to $\partial_t$; in other words, a slant curve cannot be a PA curve with $(\partial_t,\frac{\pi}{2})$. This agrees with our result in Theorem \ref{thm42}.
\end{remark}

In the next result, we consider the case $\cos \theta \neq 0$. In this case, system \eqref{trs45} yields explicit expressions for $\kappa$ and $\tau$ in terms of $f$, $\lambda_0$ and $\lambda_1$.

\begin{corollary}\label{thm43}
Let $\gamma$ be a Frenet curve in a connected orientable Riemannian $3$-manifold and $\v$ a torse-forming vector field along $\gamma$ with potential function $f$. Then $\gamma$ is a PA curve with $(\v,\theta)$, where $\cos \theta \neq 0$, if and only if $\v$ is given by \eqref{trs44} and 
\begin{equation}\label{thm41kt1}
\kappa =\frac{\lambda_0'-f(1-\lambda_0^2)}{\cos\theta}, \quad \tau=-\frac{f\lambda_0\lambda_1+\lambda_1'}{\cos\theta},
\end{equation}
where $\lambda_0$ and $\lambda_1$ are smooth functions along $\gamma$.
\end{corollary}

In the particular case that $\v$ is a concircular vector field along the curve, the following result provides a direct relation between the prescribed angle $\theta$ and the curvatures of PA curves.

\begin{corollary}\label{thm44}
Let $\gamma$ be a Frenet curve in a connected orientable Riemannian $3$-manifold and $\v$ a  concircular vector field along $\gamma$ with potential function $f$. Then $\gamma$ is a PA curve with $(\v,\theta)$, where $\cos \theta \neq 0$, if and only if $\v=\cos \theta N+\sin\theta B$ and
\begin{equation}
\kappa=-\dfrac{1}{\cos \theta } f, \quad \tau=- \theta'. \label{thm41kt2}
\end{equation}
\end{corollary}

\begin{proof}
Since $\v$ is concircular along the curve $\gamma$, its generating form vanishes identically along $\gamma$. Moreover, as $\v$ is not parallel along $\gamma$, it follows directly from \eqref{cdimm2} that $\v$ is orthogonal to $\gamma$, that is, $\lambda_0 = 0$. Then, from \eqref{trs44}, we obtain $\lambda_1 = \sin \theta$ and hence $\v=\cos \theta N+\sin\theta B$. The proof is concluded by substituting $\lambda_0 = 0$ and $\lambda_1 = \sin \theta$ into \eqref{thm41kt1}.
\end{proof}

When $f$ is constant, it is interesting to note that the curvatures in \eqref{thm41kt2} satisfy the following condition
\begin{equation}
\left (\frac{1}{\tau} \left (\frac{1}{\kappa} \right )'\right )'+\frac{\tau}{\kappa}=0, \quad \tau \neq 0. \label{geosph}
\end{equation}
This condition is important because, when the ambient space is the space form $\m^3(c)$, equation \eqref{geosph} gives a necessary and sufficient condition for Frenet curves with nonzero torsion to lie on geodesic spheres (see \cite[Theorem 4]{dgd1}). 
Recall that geodesic spheres in $\m^3(c)$ are totally umbilical hypersurfaces. In particular, in $\s^3(c)\subset \r^4$ and $\h^3(-c)\subset \r^4_1$, they are obtained as intersections of these hypersurfaces with hyperplanes.

Combining condition \eqref{geosph} with Corollary \ref{thm44}, we obtain the following result, which provides a new characterization of curves lying on geodesic spheres.

\begin{theorem}\label{thm44spherical}
An arc length parametrized Frenet curve $\gamma\subset\m^3(c)$ with nonzero torsion lies on a geodesic sphere if and only if there is a unit concircular vector field $\v$ along $\gamma$ with potential function $f=\frac{1}{a\tan \theta +b}$, where $\theta=\cos^{-1}(\langle \v,N\rangle)$ and $a,b\in \r$ are constants.
\end{theorem}
\begin{proof}
Assume that $\v$ is a unit concircular vector field along $\gamma$ with potential function $f=(a\tan \theta + b)^{-1}$, where $\cos\theta=\langle \v, N\rangle$ and $a,b\in \r$. By Definition \ref{int-def}, $\gamma$ becomes a PA curve with $(\v,\theta)$. Therefore, by \eqref{thm41kt2}, the curvatures of $\gamma$ satisfy
$$
\frac{1}{\kappa}=-(a\sin \theta + b\cos \theta), \quad \tau=\theta'.
$$
In this case, it is direct to verify that \eqref{geosph} holds. Consequently, by \cite[Theorem 4]{dgd1}, we conclude that $\gamma$ lies on a geodesic sphere.

Conversely, assume that a curve $\gamma$ in $\m^3(c)$ with nonzero torsion lies on a geodesic sphere centered at $p\in\m^3(c)$. Then its curvatures satisfy \eqref{geosph}. We divide the proof into three cases:

\textbf{Case:} $\m^3(c)=\r^3$. In this case, a geodesic sphere is a standard sphere in $\r^3$. Let $r>0$ be its radius. Then it is easy to see that
$$
\frac{1}{r}(\gamma-p)=-\frac{1}{\kappa} N-\frac{1}{\tau} \left (\frac{1}{\kappa} \right )' B.
$$
Next, define a vector field $\v=\frac{1}{r}(\gamma-p)$. Since $\nabla_{\gamma'}\v=\frac{1}{r}\gamma'$, it follows that $\v$ is a unit concircular vector field along $\gamma$ with potential function $f=\frac{1}{r}$, which corresponds to the case $a=0$ and $b=r$ in the statement of the theorem.

\textbf{Case:} $\m^3(c)=\s^3(c)\subset \r^4$. There exists an arc length parametrized geodesic $\beta$ passing through a point $p\in \s^3(c)$ such that it connects $p$ to $\gamma(s)$ and is orthogonal to $\gamma$. Denoting by $T_\beta$ the unit tangent vector field of $\beta$, we write (see \cite[Equation (27)]{dgd1})
$$
T_\beta=-\frac{1}{c}\cot\left (\frac{r}{c} \right) \left ( \frac{1}{\kappa}N+\frac{1}{\tau} \left (\frac{1}{\kappa} \right )' B \right ),
$$
where $r<c\pi/2$. Let $\nabla$ and $\nabla^0$ be the Levi-Civita connections of $\s^3(c)$ and $\r^4$, respectively. Then, the Gauss formula is written by
$$
\nabla^0_XY=\nabla_XY-\frac{1}{c^2}\langle X,Y\rangle \Phi,
$$
where $\Phi$ is the standard immersion of $\s^3(c)$ into $\r^4$. Hence, by taking the covariant derivative, we obtain  $\nabla_T T_\beta=\frac{1}{c}\cot\left (\frac{r}{c} \right) T.$ Defining the vector field $\v=T_\beta$, we see that it is a unit concircular vector field along $\gamma$ with potential function $f=\frac{1}{c}\cot\left (\frac{r}{c} \right),$ which corresponds to the case $a=0$ and $b=c\tan\left (\frac{r}{c} \right)$ in the statement of the theorem.

\textbf{Case:} $\m^3(c)=\h^3(-c)$. The proof is analogous, replacing $\cot$ with $\coth$.
\end{proof}

\begin{remark}
Notice that the potential function $f=(a\tan \theta + b)^{-1}$ in Theorem \ref{thm44spherical} arises from solving an ODE, as explained below. Assume that a PA curve in $\m^3(c)$ satisfies the hypothesis of Corollary \ref{thm44}. Then its curvatures are given by \eqref{thm41kt2}. Suppose further that this curve lies on a geodesic sphere. Substituting the expressions for the curvatures from \eqref{thm41kt2} into \eqref{geosph}, we obtain
$$
f''-\left (\frac{\theta''}{\theta'}+2\theta'\tan \theta \right)f'-2\frac{f'^2}{f}=0.
$$
The case where $f$ is constant provides a trivial solution. Assume now that $f$ is nonconstant. Introducing the change of variable $y=1/f$, we obtain
$$
y''-\left (\frac{\theta''}{\theta'}+2\theta'\tan \theta \right)y'=0.
$$
Since $\theta'=\tau\neq 0$, we may view $y$ as a function of $\theta$, that is, $y=y(\theta)$. Then the above equation reduces to
$$
\frac{d^2y}{d\theta^2}-2\tan\theta\frac{dy}{d\theta}=0.
$$
Its general solution is given by $y(\theta)=a\tan\theta+b$, where $a,b\in \r$.
\end{remark}

\begin{remark}
The proof of Theorem \ref{thm44spherical} provides a method for establishing concircular vector fields along curves. Let $\gamma$ lie on a geodesic sphere centered at $p$. If the ambient space is $\s^3(c)$ or $\h^3(-c)$, the proof shows that the tangent vector field of the geodesic joining the center $p$ to the point $\gamma(s)$ is concircular along the curve. An analogous statement also holds in $\r^3$. Indeed, the geodesic joining a point $p\in \r^3$ to $\gamma(s)$ is the ray parametrized by $t\, \mapsto p+t(\gamma(s)-p)$, $t\geq 0$. At the point $\gamma(s)$, the tangent vector to this ray is $\gamma(s)-p$, which defines a concircular vector field along the curve $\gamma$.
\end{remark}

Another consequence of Corollary \ref{thm44} is the following.

\begin{corollary} \label{lanc}
Let $\gamma$ be a PA curve with $(\v,\theta)$ in a Riemannian $3$-manifold, where $\v$ is a concircular vector field along $\gamma$ with potential function $f$ and $\cos \theta \neq 0$. Then the ratio of its torsion to curvature is a nonzero constant $r_0\in \r$ if and only if
$$\theta(s)=\sin^{-1}\left(r_0\int^sf(u)du+r_1\right), \quad r_1\in \r.$$
In particular, if $f$ is nonzero constant $f_0$, then, up to an integration constant, 
\begin{equation}
\kappa(s)=\pm\frac{f_0}{\sqrt{1-(r_0f_0s)^2}}, \quad \tau(s)=r_0 \kappa(s), \quad  -\frac{1}{|r_0f_0|}<s<\frac{1}{|r_0f_0|}. \label{conratio}
\end{equation}
\end{corollary}

\begin{proof}
The proof can be easily concluded by \eqref{thm41kt2}.
\end{proof}

In the next example, we construct such a curve in the Euclidean setting. In particular, when $f$ is constant, Theorem \ref{thm44spherical} implies that the curve lies on a standard sphere and we illustrate this case.

\begin{example}
Let $\gamma:(-1,1)\to \r^3$ be the arc length parametrized curve given by
$$
\gamma(s)=
\begin{pmatrix}
 a(s)\cos b(s)+\dfrac{s}{\sqrt{2}}\sin b(s)\\[8pt]
\dfrac{s}{2}(1-\cos b(s))+\dfrac{a(s)}{\sqrt{2}}\sin b(s) \\[8pt]
\dfrac{s}{2}(1+\cos b(s))-\dfrac{a(s)}{\sqrt{2}}\sin b(s)
\end{pmatrix},
$$
where $a(s):=\sqrt{1-s^2}$ and $b(s):=\sqrt{2}\sin^{-1}s.$ We will show that $\gamma$ is a PA curve with $(\v,\theta)$, where $\v$ is a unit concircular vector field along $\gamma$ and $\tau/\kappa$ is constant. A direct computation shows that $|\gamma(s)|=1$, and hence $\gamma\subset \s^2(1)$. Therefore, by choosing $\v=\gamma$, we obtain a unit concircular vector field along $\gamma$ with potential function $f\equiv 1$. Moreover, since $\gamma$ is parametrized by arc length, one computes that $\kappa(s)=\tau(s)=\frac{1}{a}.$ Also, $\cos\theta=\langle \v,N\rangle=-a$ and $\sin\theta=s.$ Since $f\equiv 1$, we have $\sin\theta=\int^s f(u)\,du$
up to a constant. Therefore, Corollary \ref{lanc} is verified with $r_0=1$.
\end{example}

%%%%%%%%%%%%%%%%%%%%%%%%%%%%%%%%
\section{PA curves in Riemannian surfaces} \label{secdim2}
%%%%%%%%%%%%%%%%%%%%%%%%%%%%%%%%

This section is devoted to study PA curves in Riemannian surfaces $M^2$. Let $\gamma:I\to M^2$ be a PA curve  with $(\v,\theta)$, where $\v$ is a torse-forming vector field. Then, \eqref{trs41} and \eqref{cdimm2} reduce to, respectively,
\begin{equation}
\v=\sin\theta T+\cos \theta N  \label{trs300}
\end{equation}
and
\begin{equation}
f\sin\theta+\omega(T)=0.  \label{trs301}
\end{equation}

By differentiating \eqref{trs300}, applying the Frenet formulas, and using the fact that $\v$ is a torse forming vector field, we obtain
\begin{equation}
\left. 
\begin{array}{r}
\cos \theta(f\cos \theta+\kappa-\theta') =0 \\ 
\sin \theta(f\cos \theta+\kappa-\theta')  =0.
\end{array}%
\right. \label{trs31}
\end{equation}

\begin{lemma}\label{lem31}
Let $\gamma$ be a PA curve in a Riemannian surface $M^2$ with $(\v,\theta)$, where $\v$ is a torse-forming vector field along $\gamma$ with potential function $f$ and $\cos \theta \neq 0$. Then,
\begin{equation}
\kappa=\theta'-\cos \theta f. \label{trs32}
\end{equation}
\end{lemma}

A direct consequence of Lemma \ref{lem31} follows when $\theta$ is constant. In this case, the curvature $\kappa$ is a constant multiplier of the potential function $f$ of $\v$. Another consequence of Lemma \ref{lem31} happens when $\cos \theta=\kappa$. In this case, as well as in the  Euclidean setting where $\v$ is parallel, the corresponding curves are known as grim reapers and are soliton solutions of the curve shortening flow \cite{lo1}. 

Let $e_1=(1,0)$ and $e_2=(0,1)$ be the canonical orthonormal frame of $\r^2$. Up to a rigid motion, a grim reaper satisfies $\kappa=\langle N,e_2\rangle$.  Therefore the arc length parametrization of such a curve is $\gamma(s)=(\tan^{-1}(\sinh s), \log (\cosh s)).$ The unit normal vector and the curvature are given by $N(s)=(-\tanh s,\sech s)$ and $\kappa(s)=\sech s,$ respectively (see \cite{lo1}).

In our framework, a grim reaper in $ \r^2$ can be viewed a PA curve with $(e_2,\cos^{-1}(\kappa))$. Notice here that $e_2$ is a trivially torse-forming (in fact, parallel) vector field on $\r^2$. In the following result, we prove that even when the corresponding vector field is non-parallel along the curve, the curvature function of every PA curve with $(\v,\cos^{-1}(\kappa))$ is of the same form.

\begin{theorem}\label{thm31}
Let $\gamma$ be a PA curve in a Riemannian surface $M^2$ with $(\v,\theta)$, where $\v$ is a torse-forming vector field along $\gamma$ with potential function $f$ and $\cos \theta=\kappa$. Then, up to an integration constant,
\begin{equation}
\kappa(s)=\sech\left (s+\int^s f(u) du \right) .\label{trs311}
\end{equation}
\end{theorem}

\begin{proof}
By the assumption, we have $\cos \theta=\kappa$. Substituting into \eqref{trs32} yields
$$
\theta'=\cos \theta (1+f  ).
$$
The result follows by solving this ODE.
\end{proof}

Once Theorem \ref{thm31} has been proved, we provide an example of such a PA curve in $\h^2:=\h^2(-1)$. 

\begin{example} \label{examp45}
Consider the upper-half plane model of $\h^2$ and the following curve $\gamma \subset \h^2$  given by
$$
\gamma(s)=\left (\int^s(\cosh(2s))^{\frac{-1}{2}}du, \sqrt{\cosh(2s)}\right ), \quad s\in \r.
$$
Since $\langle\gamma'(s),\gamma'(s)\rangle_h=1$, it is parametrized by hyperbolic arc length. Letting $e_1:=y\partial_x\in \x(\h^2)$ and $e_2:=-y\partial_y\in \x(\h^2)$, we have
$$
T=\sech(2s) e_1+\tanh(2s) e_2, \quad N=-\tanh(2s) e_1+\sech(2s) e_2.
$$
Let $\v=e_2$ be the unit anti-torqued vector field on $\h^2$ with conformal scalar $f=1$. Denote by $\theta $ the angle function between $e_2\,\circ \,\gamma$ and $N$. Then, it follows $\cos \,\theta(s)=\sech(2s)$. The curvature $\kappa$ of $\gamma$ is computed by the relation $\kappa=y\kappa_e+\langle N_e, \partial_y \rangle$, where $\kappa_e$ and $N_e$ are the curvature and the unit normal with respect to the Euclidean metric $\langle , \rangle$ (see \cite[Equation (4)]{kl}). Therefore, a computation gives $\kappa(s)=\sech(2s)$.  Consequently, $\gamma$ becomes a PA curve in $\h^2$ with $(e_2,\theta)$, where $ \kappa=\cos \theta$.
\end{example}

\begin{remark}
Although the curve in Example \ref{examp45} satisfies the relation $\kappa=\langle e_2,N\rangle,$
it is not a soliton solution to the curve shortening flow in $\h^2$. This is because, in the hyperbolic plane, soliton solutions of the curve shortening flow are characterized by $\kappa=\langle N,X\rangle,$
where $X\in \x(\h^2)$ is a Killing vector field generating a one-parameter group of isometries of $\h^2$ (see \cite{kl}). However, the vector field $e_2$ is not a Killing vector field on $\h^2$. Therefore, the relation $\kappa=\langle e_2,N\rangle$
does not describe a soliton of the curve shortening flow in $\h^2$, unlike the Euclidean grim reaper equation.
\end{remark}

%%%%%%%%%%%%%%%%%%%%%%%%%%%%%%%
\section*{Acknowledgments}
%%%%%%%%%%%%%%%%%%%%%%%%%%%%%%%
This study was supported by Scientific and Technological Research Council of Turkey (TUBITAK) under the Grant Number (123F451). The authors thank to TUBITAK for their supports.

%%%%%%%%%%%%%%%%%%%%%%%%%%%%%%%
\section*{Declarations}
%%%%%%%%%%%%%%%%%%%%%%%%%%%%%%%

\begin{itemize}
\item Conflict of interest: The authors declare no conflict of interest.

\item Ethics approval: Not applicable.

\item Availability of data and materials: Not applicable.

\item Code availability: Not applicable.

\item Authors’ contributions: Muhittin Evren Aydın formulated the problem, performed the main analysis, and designed the structure of the manuscript. Esra Dilmen and Büşra Karakaya contributed to solving the problem and to writing the manuscript. All authors read and approved the final version of the paper.
\end{itemize}

%%%%%%%%%%%%%%%%%%%%%%%%%%%%%%%

 \end{document}